\documentclass[12pt,leqno]{article}
\usepackage[dvips]{graphicx}
\usepackage{amsmath,amssymb,amsthm,amscd}
\usepackage[mathscr]{eucal}
\usepackage{amsfonts}
\begin{document}
\parskip=6pt
\newtheorem{prop}{Proposition}
\numberwithin{prop}{section}
\newtheorem{theorem}{Theorem}
\newtheorem{corr}{Corollary}
\newtheorem{lemma}{Lemma}
\newtheorem{defn}{Definition}
\numberwithin{defn}{section}
\numberwithin{lemma}{section}
\numberwithin{equation}{section}
\newtheorem{proof2}{Proof of Theorem 1.2}
\newtheorem{proof3}{Proof of Theorem 1.3}
\newtheorem{proof4}{Proof of Theorem 1.4}

\newcommand{\sD}{{\cal D}}
\newcommand{\sC}{{\cal C}}
\newcommand{\sP}{{\Cal P}}
\newcommand{\sF}{{\cal F}}
\newcommand{\sG}{{\Cal G}}
\newcommand{\sL}{{\cal L}}
\newcommand{\sH}{{\cal H}}
\newcommand{\sR}{{\Cal R}}
\newcommand{\sS}{{\cal S}}
\newcommand{\sA}{{\cal A}}
\newcommand{\sE}{{\cal E}}
\newcommand{\sQ}{{\Cal Q}}
\newcommand{\sM}{{\cal M}}
\newcommand{\sB}{{\cal B}}
\newcommand{\sK}{{\cal K}}
\newcommand{\loc}{\text{loc}}
\newcommand{\sN}{{\eta}}
\newcommand{\bb}{\mathbf b}
\newcommand{\bo}{\mathbf o}
\newcommand{\bx}{\mathbf x}
\newcommand{\bp}{\mathbf p}
\newcommand{\bz}{\mathbf z}
\newcommand{\bs}{\mathbf s}
\newcommand{\bl}{\mathbf l }
\newcommand{\by}{\mathbf y}
\newcommand{\bO}{\mathbf O}
\newcommand{\oy}{\overline{y}}
\newcommand{\ox}{\overline{\mathbf x}}
\newcommand{\bR}{\mathbb{R}}
\newcommand{\bS}{\mathbb{S}}
\newcommand{\bC}{\mathbb{C}}
\newcommand{\ds}{\displaystyle}
\newcommand{\var}{\varepsilon}


\title{Regularity and the Behavior of Eigenvalues for Minimizers of a Constrained $Q$-tensor Energy for Liquid Crystals}
\author{
Patricia Bauman\thanks{Research supported by NSF grant DMS-1412840}\\
Daniel Phillips\thanks{Research supported by NSF grant DMS-1412840}\\
Department of Mathematics\\
Purdue University\\
West Lafayette, IN\ 47907\\
bauman@math.purdue.edu\\
phillips@math.purdue.edu}

\maketitle

\begin{abstract}
We investigate minimizers defined on a bounded domain in $\bR^2$ for the Maier--Saupe Q--tensor energy used to characterize nematic liquid crystal configurations.  The energy density  is singular, as in  Ball and Mujamdar's modification of the Ginzburg--Landau Q--tensor model, so as to constrain the competing states to have eigenvalues in the closure of a physically realistic range.
We prove that minimizers are regular and in several model problems we are able to use this regularity to prove that minimizers have eigenvalues strictly within the physical range.
\end{abstract}
\section{Introduction}

In this paper we prove regularity properties and  bounds on the eigenvalues for local minimizers of an energy derived from Maier--Saupe theory, a model that is used to describe order in
nematic liquid crystal materials. (See \cite{MS}.)
We examine the special case of  a liquid crystal occupying a cylindrical region in $\bR^3$ with  cross--section $\Omega\subset\bR^2$, where $\Omega$ is an
open bounded domain with a smooth $(C^2)$ boundary.
The liquid crystal material is described at almost every $(x_1,x_2)$ in $\Omega$ by a $3\times 3$ matrix $Q(x_1,x_2)$ and is assumed to be uniform in the $x_3$--direction.
The function $Q$ is a tensor--valued order parameter, by which we mean that for almost every $\bx=(x_1,x_2)$ in $\Omega$,
$$
Q(\bx)\in\sS_0:=\{Q\in\bR^{3\times 3}\colon Q=Q^t\text{ and }tr\ Q=0\}.
$$
where $\bR^{3\times 3}$ denotes the space of $3\times3$ real--valued matrices.
The energy considered here is:

\begin{equation}
{\sF}(Q)=\int_\Omega [f_e (Q(\bx),\nabla Q(\bx))+f_b (Q(\bx))]d\bx.
\end{equation}
To describe our assumptions on the energy density  of $\sF$, let $Q\in\sS_0$ and let $\lambda_1(Q)\leq\lambda_2(Q)\leq\lambda_3(Q)$ denote its eigenvalues.
Define the set $\sM\subset\sS_0$ as the set of matrices $Q\in\sS_0$ for which $\lambda_i(Q)\in (-{1\over 3},{2\over 3} )$ for all $1\leq i\leq 3$.
Then $\sM$ is an open, bounded, and convex subset of $\sS_0$. Note that $\partial\sM$ is the set of all $Q$ in $\sS_0$ such that $\lambda_i(Q)\in [-{1\over 3},{2\over 3} ]$ for all $1\leq i\leq 3$ and $\lambda_j(Q)\in\{-{1\over 3},{2\over 3}\}$ for at least one $j$.
We assume throughout the paper that $f_b$ satisfies
\begin{equation}
f_b(Q) = \left\{
    \begin{array}{l l}f(Q)-\kappa |Q|^2+b_0\quad\text{ for }Q\in\sM,\\
             +\infty\hskip .9in \text{ for }Q\in\sS_0\backslash\sM.\\
    \end{array} \right.
 \end{equation}
where $\kappa$ and $b_0$ are constants, $\kappa\geq 0$, $f$ is convex, $f\in C^\infty(\sM)$, and
$\underset{ Q\to\partial\sM
} {\lim} f(Q)=+\infty$.
Such a function is bounded below and hence we assume without loss of generality that
$$
b_0=-\underset{Q\in \sM}{\min} \{f(Q)-\kappa |Q|^2\}
$$
so that $f_b\geq0$.  Our assumptions on the bulk energy density $f_b$ are motivated by the work in the papers \cite{BM1} by Ball and Mujamdar and \cite{KK} by Katriel, Kventsel, Luckhurst, and Sluckin. In these papers  a particular potential function  $f(Q)=f_{ms}(Q)$ as described
above was constructed
consistent with Maier--Saupe theory.
More precisely, it was proved in \cite{BM1} and  \cite{KK} that for each $Q\in\sM$ there is a unique function $\hat{\rho}=\hat{\rho}(Q)$ in the set

\[ A_Q:=\{\rho\in L^1 (\bS^2;\bR): \rho\geq 0, \ds\int_{\bS^2}\rho(\bp) \ d\bp=1,
Q=\int_{\bS^2} (\bp\otimes\bp-{1\over 3}\ I)\rho (\bp) d\bp\}\]

satisfying
\begin{equation}
f_{ms}(Q):=\underset{\rho\in A_Q}{\inf}\int_{\bS^2} \rho (\bp)\ln (\rho (\bp)) d\bp=\int_{\bS^2}\hat{\rho}(\bp)\ln (\hat{\rho} (\bp)) d\bp.
\end{equation}
Given $Q\in\sM$ the set $A_Q$ is the family of  (absolutely continuous)  probability densities $\rho$ with $Q$ as their normalized  second moments. The set $A_Q$ is nonempty for $Q$ in $\sS_0$ if and only if $Q\in\sM$; if $Q\in\partial\sM$ one can still represent it as the second moment of a density,
 \[Q=\int_{\bS^2} (\bp\otimes\bp-{1\over 3}\ I) d\mu(\bp),\]
 however in this case the density $\mu$ will be a singular measure and this is not considered  physical. (See \cite{BM1}.)  The densities in $A_Q$ then provide all possible physically admissible statistics for the local orientation of the liquid crystal molecules  with  second moments given by $Q$. It is proved in  \cite{BM1} and \cite{KK} that $f_{ms}(Q)$ has the properties postulated above for $f(Q)$ except for the fact that $f_{ms}\in C^\infty(\sM)$. This fact is proved in \cite{F}.

We next describe our assumptions on the elastic energy density $f_e$.
Let
\begin{eqnarray*}
\sD&:=&\{D=[D_{ijk}]\ 1\leq i,j,k\leq 3:\\
D_{ijk}&=& D_{jik}\text{ and }\sum^3_{\ell=1} D_{\ell\ell k}=0\text{ for each } i,j,\text{ and }k\}.
\end{eqnarray*}
Note  that $\sD$ includes the tangent space for differentiable mappings $Q:\Sigma\rightarrow\sS_0$ where $\Sigma$ is an open subset of $\bR^3$, i.e., $[\partial_{x_k} Q_{ij}(\bx)]\in\sD$  for each differentiable $\sS_0$--valued $Q$ and each $\bx$ in $\Sigma$.
We assume that $f_e(Q,D)$ is continuous on $\overline{\sM}\times\sD$ and that there are constants $0<\alpha_1\leq\alpha_2<\infty$, $0\leq M_1\leq M_2<\infty$ so that
\begin{equation}
\alpha_1 |D|^2-M_1\leq f_e (Q,D)\leq \alpha_2 |D|^2 +M_2
\end{equation}
for all $(Q,D)\in\overline{\sM}\times\sD$.
 A specific class of elastic energy density functions that we have in mind are the {\it Landau--de Gennes elastic energy densities} $f_{ld}(Q,D)$ defined as polynomials in terms of the components of $Q$ and $D$ that are
$SO(3)$ invariant, that is, $f_{ld}(Q,D)=f_{ld}(Q^*,D^*)$ for every $(Q,D)$ in $\sS_0\times\sD$ and $R$ in $SO(3)$ where
\begin{eqnarray*}
&Q^*_{ij}=R_{i\ell}Q_{\ell m}R_{jm},\\
&  D^*_{ijk}=R_{i\ell}R_{jm}R_{kh}D_{\ell mh}.
\end{eqnarray*}
Here we used the summation convention for repeated indices among $\ell,m,h$ in the set $\{1,2,3\}$.
We remark that for the case of a Landau--de Gennes elastic energy density  which is a function of $D$ only, denoted by $f_{ld}(Q,D)=f_{ld}^{(1)}(D)$, it is known that  $f_{ld}^{(1)}$ satisfies (1.4) if and only if for any differentiable function $Q(x_1,x_2,x_3)$ valued in $\sS_0$ we have
\begin{eqnarray*}
f_{ld}^{(1)}(\nabla Q)&=&L_1 Q_{ij,x_\ell}Q_{ij,x_\ell}+L_2 Q_{ij,x_j}Q_{ik,x_k}\\
&+&L_3 Q_{ij,x_k}Q_{ik,x_j}
\end{eqnarray*}
with the elasticity constants satisfying

\[L_1+\frac{5}{3}L_2+\frac{1}{6}L_3>0,\,
L_1-\frac{1}{2}L_3>0,\,
L_1+L_3>0.\]
 (See \cite{LMT}.) This is a classic example among Landau--de Gennes models.
More general examples of Landau de Gennes energy densities are also given in \cite{LMT} such as
\begin{eqnarray*}
f_{ld}^{(2)}(Q,\nabla Q)&=&f_{ld}^{(1)}(\nabla Q)+L_4\varepsilon_{lkj} Q_{li}Q_{ki,x_j}\\
&+& L_5 Q_{lk} Q_{ij,x_l} Q_{ij,x_k}
\end{eqnarray*}
where $\varepsilon_{\ell kj}$ is the Levi--Civita tensor.
The $L_4$ term above allows the model to account for molecular chirality, and cubic expressions such as the $L_5$ term permit a more specific description of elastic
contributions.
Since $\overline\sM$ is a bounded set one can easily give conditions on the elasticity constants so that $f_{ld}^{(2)}$ satisfies (1.4).
(See (1.10).)
Note that in this paper we apply the energy density $f_{ld}(Q,D)$ to functions $Q$ in $H^1_{loc}(\Omega, \overline\sM)$ where $\Omega$ is a subset of the $ x_1 x_2$ plane, and as such for the functions considered here we have $D_{ij3}=Q_{ij,x_3}=0.$

We consider fixed boundary conditions,
\begin{equation}
Q=Q_0\in C^{0,1} (\partial\Omega;\overline\sM) \text{ such that } \int_{\partial\Omega}f_b(Q_0) ds<\infty.
\end{equation}
Set
$$
\sA_0=\{Q\in H^1 (\Omega;\overline\sM)\colon Q=Q_0\text{ on }\partial\Omega\}.
$$
One can show that there exists $Q\in\sA_0$ such that $\mathcal{F}(Q)<\infty$.

In this paper, under the assumptions described above on $\Omega, f_e, f_b$, and $Q_0$, we analyze finite energy local minimizers of (1.1) and global minimizers in $\sA_0$.
Note that a Landau--de Gennes elastic energy density $f_e=f_{ld}(Q,D)$ satisfies (1.4) if and only if it is of degree two in $D$ for each $Q$ and the mapping $D\in\sD\to f_{ld}(Q,D)$ is strictly
convex, uniformly in $Q$ for $Q\in\sM$.
In light of this, if $f_e=f_{ld}$ the existence of global minimizers in $\sA_0$ follows in a standard way by using direct methods, just as in \cite{BM1}.

Our main results can be described as follows. In Section 2 we first prove that local minimizers of $\mathcal{F}$ in $H^1(\Omega;\overline\sM)$ are uniformly H\"older continuous in subdomains whose closure is a compact subset of $\Omega$.
This follows immediately from:
\begin{theorem} Assume $r_0>0$, $0<r<r_0$, and $B_{4r}(\bo)$ is an open disk contained in $\Omega$ with center $\bo$ and radius $4r$. Let $B_s=B_s(\bo)$. Assume $Q\in H^1 (B_{4r};\overline\sM)$ such that
$$
\mathcal{F}(Q;B_{4r})\colon =\int_{B_{4r}} [f_e (Q,\nabla Q)+f_b (Q)]d\bx<\infty
$$
and
$$
\mathcal{F}(Q;B_{4r})\leq \mathcal{F}(V;B_{4r})
$$
for all $V\in H^1 (B_{4r};\overline\sM)$ such that $V-Q\in H_0^1 (B_{4r};\sS_0)$.
Then there exists a constant $\sigma$ in $(0,1)$ and $c>0$  depending only on  $r_0$ and on the constants in (1.2) and (1.4) such that $Q\in C^\sigma(\overline B_r)$ and
$$
\|Q\|_{C^\sigma(\overline B_r)}\leq \frac{c}{r^\sigma}\sqrt{(\mathcal{F}(Q; B_{2r})+1)}.
$$
\end{theorem}
We then consider local minimizers near $\partial\Omega$. Let $r_1$ be a positive constant such that for each point $\bo$ on $\partial\Omega$, $B_{r_1}(\bo)\cap\Omega$ is a coordinate neighborhood for the $C^2$ structure of $\partial\Omega$ and is diffeomorphic to a half-disk.

\begin{theorem} Assume  $\bo$ is any point in $\partial\Omega$ and let $B_s=B_{s}(\bo)$.
There exists a constant $\sigma$ in $(0,1)$ and $c>0$  depending only on $\Omega$, $r_1$, $Q_0$, and the constants in (1.2) and (1.4), so that if $0<r\leq r_1, \quad Q\in H^1 (\Omega\cap B_{4r};\overline\sM),\quad Q=Q_0\,\text{on}\, B_{4r}\cap\partial\Omega,\quad \mathcal{F}(Q;\Omega\cap B_{4r}) <\infty$,
and
$$
\mathcal{F}(Q;\Omega\cap B_{4r})\leq \mathcal{F}(V;\Omega\cap B_{4r})
$$
for all $V\in H^1 (\Omega\cap B_{4r};\overline\sM)$ such that $V-Q\in H_0^1 (\Omega\cap B_{4r};\sS_0)$, then $Q\in C^\sigma (\overline{\Omega\cap B_r})$ and
$$
\|Q\|_{C^\sigma(\overline{\Omega\cap B_r})}\leq\frac{c}{r^\sigma}\sqrt{(\mathcal{F}(Q; B_{2r})+1)}.
$$
\end{theorem}

In Maier--Saupe theory the set
\begin{eqnarray*}
\Lambda(Q)&=&\{\bx\in\overline{\Omega}\colon Q(\bx)\in\partial\sM\}\\
&\equiv&\{\bx\in\overline{\Omega}\colon \lambda_i(Q(\bx))\in[-\frac{1}{3},
\frac{2}{3}] \text{ for } 1\leq i \leq 3 \text{ and }\\
&&\lambda_j( Q(\bx))\in\{-\frac{1}{3},
\frac{2}{3}\} \text{ for some } j \in \{1,2,3\}\}
\end{eqnarray*}
corresponds to locations where perfect nematic order occurs and this is interpreted as  not physical.
(See \cite{BM1}.)
Note that by (1.2), if $\mathcal{F}(Q;E)<\infty$ then $\Lambda(Q)\cap E$ has measure zero.
Thus we have:

\begin{corr}Let $Q$ be a minimizer of $\sF$ in $\sA_0$.
Then $Q$ is H\"{o}lder continuous in $\overline\Omega$ and hence $\Lambda(Q)$ is compact.
\end{corr}
We next assume that $f_e=f_{ld}(Q,D)$ (and hence from our assumptions $f_{ld}$ satisfies (1.4)).
The condition $Q\in\sS_0$ allows us to express the nine components of $Q$ in terms of the five independent variables
\begin{equation}
\bz=(z^1,\ldots,z^5)=(Q_{11},Q_{12},Q_{13},Q_{22},Q_{23}),
\end{equation}
as $Q=\tilde Q (\bz)$, so that $Q(\bx)=\tilde Q(\bz(\bx)).$ The energy takes the form
\begin{eqnarray}
{\sF'}(\bz)&=&\int_\Omega [a_{ij}^{lk}(\bz)\partial_{x_l} z^i \partial_{x_k}z^j+b_i^\ell(\bz)\partial_{x_l}z^i\\
&+& f_b (\tilde Q(\bz))]d\bx\nonumber
\end{eqnarray}
where the coefficients are polynomials and there is a constant $\lambda>0$ so that
\begin{eqnarray}
&& a_{ij}^{lk}(\bz)\zeta^i_l\zeta_k^j\geq\lambda |\zeta|^2\text{ for all}\\
&& \zeta\in\bR^{5\times 2}\text{ and }\bz\text{ such that }\tilde Q(\bz)\in\overline{\sM}.\nonumber
\end{eqnarray}
From elliptic regularity theory ( see \cite{G} and \cite{C} ) we obtain the following:

\begin{corr} Assume $f_e=f_{ld}(Q,D)$ and $f_{ld}$ satisfies our assumptions. Then:
\begin{enumerate}
   \item If $B$ is an open disk contained in $\Omega$, a finite energy local minimizer $Q$ in  $H^1(B;\overline\sM)$  for $\sF$  satisfies $Q\in C^\infty (B\setminus \Lambda(Q))$.
   \item If $Q$ and  $B=B_{4r}(\bo)$ are as in Theorem 2 such that $B\cap\Lambda(Q)=\emptyset$, then $Q$ is as smooth in $B_r(\bo)\cap \overline\Omega$ as $\partial\Omega$ and $Q_0$ allow.
In particular if  $B\cap\partial\Omega$ is of class $C^{k,\alpha}$ and $Q_0\in C^{k,\alpha}(B\cap\partial\Omega)$ for some $k\geq 2$ and $0<\alpha<1$, then $Q\in C^{k,\alpha}(B_r(\bo)\cap\overline\Omega)$.
 \end{enumerate}

\end{corr}
In Section 3 we analyze $\Lambda(Q)$ for local minimizers, in the case
\begin{eqnarray}
f_e(Q,\nabla Q)&=&L_1 |\nabla Q|^2 + L_4 \varepsilon_{lkj} Q_{\ell i} Q_{k i,x_j}\\
&+& L_5\ Q_{lk}\ Q_{i j,x_l}\ Q_{i j,x_k}.\nonumber
\end{eqnarray}
If $(Q,D)\in\overline\sM\times\sD$, we have
\[
-{1\over 3} |D|^2\leq Q_{lk} D_{ijl} D_{ijk}\leq {2\over 3} |D|^2.
\]
Thus the natural ellipticity condition for (1.9) is
\begin{eqnarray}
L_1-{L_5\over 3}>0&\qquad\text{ if }L_5 \geq 0,\\
L_1+{2 L_5\over 3}>0&\qquad\text{ if }L_5<0.\nonumber
\end{eqnarray}
It follows that if $L_1$ and $L_5$ satisfy (1.10) and $L_4$ is fixed then (1.4) holds for the elastic density (1.9) for some positive constants $\alpha_1,\alpha_2,M_1,M_2$. The uniform convexity of the elastic density $f_e(Q,D)$ in $D$ is  due to  having  $Q$ constrained to  values in $\overline\sM$. This structure leads to the existence of minimizers for $\sF$ over $\sA_0$ with $f_e$ as in (1.9) using direct methods. In contrast to this setting, in \cite{BM1} Ball  and Majumdar  examined the classical unconstrained (Landau-de Gennes) energy that included the case where $f_e$ is as in (1.9) and $f_b$  is a polynomial in $Q$. They studied the problem of minimizing the energy  over $H^1(\Omega;\sS_0)$ with $Q=Q_0$ on $\partial\Omega$. They proved that if $L_5\neq 0$ then the energy is not bounded below and  the unconstrained minimum problem is ill-posed. Thus the unconstained energy is ill suited for investigating certain relevant elastic effects. These observations  motivate  investigating the minimum problem for the Maier-Saupe energy. Minimizers for the constrained problem, however, may a priori take on un-physical states. We show that if $f_e$ is as in (1.9) then this is impossible.
We prove the following physicality result.

\begin{theorem}Let $Q$ be a finite energy local minimizer for $\sF(\cdot;B)$ for a ball $B\subset\Omega$  where $f_e$ is as in (1.9) satisfying (1.10). Then $\Lambda(Q)\cap B=\emptyset$ and we have $Q\in C^\infty(B)$.
\end{theorem}

Theorem 3 generalizes a result of Ball and Majumdar \cite{BM1} proved for the case in which  $\Omega$ is replaced by a domain $D$ in $\bR^3$, the elasticity density $f_e(\nabla Q)=L_1\ |\nabla Q|^2$, and  $Q_0$ is valued in $\mathcal{M}$. They  prove that minimizers $Q$ in $\sA_0$ are in $C^\infty(D)$ and have $\Lambda(Q)=\emptyset$.  In this paper we extend their result to local minimizers defined on a domain $\Omega\subset \bR^n$ based on our approach.  (See Theorem 4 and Corollary 3.) Further prior work concerning regularity results for the constrained minimum problem has been done by
  L. Evans, O. Kneuss, and H. Tran  in \cite{E1} where they investigated minimizers for energies of the form (1.1) assuming that $\Omega\subset\bR^n$ for $n\geq2$ and proved partial regularity results.

\section {H\"{o}lder Continuity and Higher Regularity}

Our principal technique  is to use elliptic replacements as a way of constructing comparison functions valued in $\overline\sM$.

\begin{defn}Let $Q\in H^1(\Omega;\overline\sM)$ and let $B$ be an open ball in $\bR^2$ such that $B\cap\Omega\neq\emptyset$.
Set
\begin{equation*}
Q_B(\bx)= \left\{
    \begin{array}{l l}
   Q(\bx),\hskip .8in\text{ for }\bx\in\Omega\backslash B, \\
H (\bx)=[H_{ij}(\bx)],\text{ for }\bx\in B\cap\Omega\text{ and }1\leq i,j\leq 3,
\end{array}\right.
\end{equation*}
where
\begin{eqnarray*}
\Delta H_{ij}&=&0\qquad\text{in }B\cap\Omega,\\
H_{ij}&=&Q_{ij}\qquad\text{on }\partial(B\cap\Omega).
\end{eqnarray*}
\end{defn}

\begin{lemma}
 $Q_B(\bx)\in\overline{\sM}$ for all $\bx\in\Omega$.
In addition, $Q_B(\bx_0)\in\partial \sM$ for some $\bx_0\in B\cap\Omega$ if and only if $Q_B(\bx)\in\partial\sM$ for all $\bx$ in the component of $B\cap\Omega$ containing $\bx_0$.
\end{lemma}

\begin{proof} Let $tr A$ denote the trace of the square matrix $A$. We have that $tr\ Q_B(\bx)=0$ and $Q_B(\bx)=Q_B^t (\bx)$ in $H^{1/2}(\partial[B\cap\Omega])$.
Thus these equalities hold for all $\bx\in B\cap\Omega$ since $Q_B$ is harmonic in $B\cap\Omega$.
It follows that $Q_B(\bx)\in\sS_0$ for all $\bx\in B\cap\Omega$.
Next we define
\[
g(\bx;\zeta)=\zeta^t Q_B(\bx)\zeta\,\text{ for each }\bx\in B\cap\Omega \text{ and }\zeta\in\bR^3,\,|\zeta|=1.
\]
For each $\zeta$ this is a harmonic function in $\bx$.
We have that $Q(\bx_0)\in\partial\sM$ for some $\bx_0\in B\cap\Omega$ if and only if $g(\bx_0;\zeta_0)={2\over 3}$ or $-{1\over 3}$ for some $\zeta_0$.
  Since$-{1\over 3}\leq g(\bx;\zeta_0)\leq {2\over 3}$ then by the strong maximum principle such an $\bx_0$ exists if and only if $Q_B(\bx)\in\partial\sM$ for all  $\bx$ in the component of $B\cap\Omega$ containing $\bx_0$.
\end{proof}
 Recall that if $D$ is an open subset of $\bR^2$ with a piece-wise smooth boundary then a function $v\in H^1(D)$ is defined pointwise $\sH^1$-almost everywhere on $\overline D$. (See \cite{E}.)
\begin{defn} Let $U$ be an open set in $\bR^2$ with a piecewise smooth boundary and let $E\subset\overline U$.  A function $Q\in H^1(U;\overline\sM)$ is defined to be  separated from $\partial\sM$ on $E$ if and only if $\underset{\bx\in E}{ess}\inf\text{ dist}(Q(\bx),\partial\sM)>0$ where the essential infimum is taken with respect to $\sH^1$-measure.
\end{defn}

\noindent
{\it Proof of Theorem 1}.
Let $\bx_0\in B_r(\bo)$, $0<\rho\leq r$, ${\rho\over 2}\leq s\leq\rho$, and set $Q_s(\bx) := Q_{B_s(\bx_0)}(\bx)$.
By (1.2), $f_b(Q_s)=f_b(H)\leq f(H)+b_0$ in $B_s(\bx_0).$ Using (1.4) and the fact that $|Q|$ is uniformly bounded almost everywhere in $B_{4r}(\bo)$ (since $Q$ is valued in $\overline\sM$ a.e.) we have
\begin{eqnarray}
&&\int_{B_{\rho/2}(\bx_0)} [\alpha_1|\nabla Q|^2 + f_b (Q)]d\bx\leq {\sF}(Q;B_{2\rho}(\bx_0))+C_0\rho^2\\
&&\leq {\sF}(Q_s;B_{2\rho}(\bx_0))+C_0\rho^2\nonumber\\
&&\leq {\sF}(Q_s;B_{s}(\bx_0))+{\sF}(Q;B_{2\rho}(\bx_0)\backslash B_s(\bx_0))+C_0\rho^2\nonumber\\
&&\leq \alpha_2\int_{B_s(\bx_0)} |\nabla H|^2d\bx+\int_{B_s(\bx_0)} f(H)d\bx+{\sF}(Q;B_{2\rho}(\bx_0)\backslash B_s(\bx_0))\nonumber\\
&&\qquad+ C_1 \rho^2.\nonumber
\end{eqnarray}
where $C_1$ depends only on $M_1, M_2,\text{ and } b_0.$ Let $\overline Q$ be the average of $Q$ on $B_{2\rho}(\bx_0)\backslash B_s(\bx_0)$.
Since $H$ is harmonic and $H=Q$ on $\partial B_s(\bx_0)$ we have
\begin{eqnarray}
\int_{B_s(\bx_0)}|\nabla H|^2d\bx&= &\int_{B_s(\bx_0)}|\nabla (H-\overline Q)|^2d\bx\\
&\leq& C_2\|Q-\overline Q\|^2_{H^{1/2}(\partial B_s(\bx_0))}\nonumber\\
&\leq & C_3\|Q-\overline Q\|^2_{H^1(B_{2\rho}(\bx_0)\backslash B_s(\bx_0))}\nonumber\\
&\leq&C_4\int_{B_{2\rho}(\bx_0)\setminus B_{s}(\bx_0)}(|\nabla Q|^2+\rho^{-2}|Q-\overline Q|^2)d\bx\nonumber\\
&\leq& C_5\int_{B_{2\rho}(\bx_0)\setminus B_{s}(\bx_0)}|\nabla Q|^2d\bx\nonumber
\end{eqnarray}
where  the $C_i$ are independent of $s, \rho$, and $Q$ for ${\rho\over 2}\leq s\leq\rho$.
We next take $s$ so that
\begin{equation}
\int_{\partial B_s(\bx_0)} f_b(Q)d\bs < \infty.
\end{equation}
where $d\bs$ denotes arc-length. This holds for almost every ${\rho\over 2}\leq s\leq\rho$. Note that (2.3) holds if and only if $Q(\bx)\in\sM$ almost everywhere in $\partial B_s(\bx_0)$ and $f(Q)\in L^1(\partial B_s(\bx_0)).$
If $Q$ is separated from $\partial\sM$ on $\partial B_s(\bx_0)$ then $f$ is $C^\infty$ on a neighborhood of the range of $H|_{B_s(\bx_0)}$. Let $Q=\tilde Q(\mathbf z)$ as in (1.6). It follows that  $f(\tilde Q(\bz))$ is a  convex function of $\bz$ and that it is a smooth function for all $\bz$ such that $\tilde Q(\bz)$ is in this neighborhood. Thus $f(H(\bx))=f(\tilde Q(\bz(\bx))),$ with $\bz(\bx)=(z_1,\cdots,z_5)(\bx)=(H_{11},H_{12},H_{13},H_{22},H_{23})(\bx)$,
is a classical subharmonic function on $B_s(\bx_0)$.
Indeed
\begin{eqnarray*}
\Delta f(H(\bx))&=&\Delta [ f(\tilde Q(\bz(\bx)))]\\&=& \partial_{z_l} [f(\tilde Q(\bz)]|_{\bz(\bx)} \Delta z_l(\bx)+\partial_{z_{l}z_{k}}^2 f(\tilde Q(\bz))|_{\bz(\bx)} \partial_{x_i} z_{l}\partial_{x_i} z_{k}
\geq0.
\end{eqnarray*}
It follows from this and the mean value theorem that
\begin{equation}
\int_{B_s(\bx_0)} f( H)d\bx\leq {s\over 2}\int_{\partial B_s(\bx_0)} f(Q)d\bs.
\end{equation}
If $Q$ is not separated from $\partial\sM$ on $\partial B_s(\bx_0)$ we consider $\tau H(\bx)$ for $\tau\uparrow 1$ and $\bx\in B_s(\bx_0)$.
By the maximum principle, since $H=Q$ on $\partial B_s(\bx_0),$ we have for $0<\tau<1$ that
$$
-{1\over 3}< {-\tau\over 3}\leq\tau\underset{\begin{subarray} \, \mathbf{x} \in B_s(\bx_0) \\ |\zeta|=1\end{subarray} } \inf \zeta^t H(\bx)\zeta\leq\tau
\underset{\begin{subarray} \, \mathbf{x} \in B_s(\bx_0) \\ |\zeta|=1\end{subarray} }\sup \zeta^t H(\bx)\zeta
\leq {2\tau\over 3} < {2\over 3}.
$$
Thus $\tau H\in \sM$ on $B_s(\bx_0),$ $\tau Q\in \sM$  almost everywhere on $\partial B_s(\bx_0)$, and $f(\tau Q)\in L^1(\partial B_s(\bx_0))$ for each $0<\tau<1$, and we get
\begin{equation*}
\int_{B_s(\bx_0)} f(\tau  H)d\bx\leq {s\over 2}\int_{\partial B_s(\bx_0)} f(\tau Q)d\bs.
\end{equation*}
Note that since $f$ is convex we have
\[
f(\tau Q)\leq\tau f(Q)+(1-\tau) f(O) \quad\text{  for all  } Q\in\sM.
\]
As $f$ is bounded below on $\sM$ we can apply the dominated converges theorem to the right side and Fatou's lemma to the left as $\tau\uparrow 1$.
Thus (2.4) holds for all $s$ such that (2.3) is true.

Next choose  a constant $C_6=C_6 (b_0,\kappa)$ and a value $\frac{\rho}{2}\leq\overline{s}\leq\rho$ so that
\begin{eqnarray}
{\overline s\over 2}\int_{\partial B_{\overline s}(\bx_0)} f(Q)d\bs&\leq& {\overline s\over 2}\int_{\partial B_{\overline s}(\bx_0)} f_b(Q)d\bs+C_6\rho^2\\
&\leq&\int_{B_\rho(\bx_0)\backslash B_{\rho/2}(\bx_0)}f_b(Q)d\bx + C_6\rho^2.\nonumber
\end{eqnarray}
Set $s=\overline s$ in (2.2) and (2.4); using these statements with (2.1) and (2.5) we arrive at
\begin{eqnarray*}
\mu_1\int_{B_{\rho/2}(\bx_0)} [|\nabla Q|^2+f_b (Q)]d\bx&\leq&\mu_2\int_{B_{2\rho(\bx_0)\backslash B_{\rho/2}(\bx_0)}}
[|\nabla Q|^2+f_b (Q)]d\bx\\
&+& C_7\rho^2
\end{eqnarray*}
where $\mu_1,\mu_2$, and $C_7$ are positive constants depending only on the constants in (1.2) and (1.4).
The above inequality allows us to do a "hole filling" argument. Adding $\mu_2\ds\int_{B_{\rho/2}}[|\nabla Q|^2 + f_b (Q)]d\bx$ to both sides we obtain
\begin{eqnarray*}
\int_{B_{\rho/2}(\bx_0)} [|\nabla Q|^2+f_b (Q)]d\bx&\leq&\theta\int_{B_{2\rho}(\bx_0)} [|\nabla Q|^2+f_b (Q)]d\bx\\
&+& C_8\rho^2
\end{eqnarray*}
where $\theta=\ds{\mu_2\over\mu_1+\mu_2} < 1$ for all $\bx\in B_r(\bo)$ and $\rho\leq r \leq r_0.$
Iterating this inequality and setting
\[
w_{\bx_0}(\rho)=\int_{B_\rho(\bx_0)}[|\nabla Q|^2 + f_b(Q)] d\bx
\]
we obtain the Morrey-type estimate
\[
w_{\bx_0}(\rho)\leq C(\frac{\rho}{r})^{2\sigma}[w_{\bx_0}(r)+r^{2\sigma}]
\]
for all $\bx_0\in B_r(\bo)$ and $\rho\leq\frac{r}{2}\leq\frac{r_0}{2}$ where $\sigma>0$ depends on $\theta$,
and $C$ depends on $\theta, r_0,$ and $C_8$. Since $Q$ is bounded, the theorem follows from this and Morrey's theorem. (See \cite{G}, Ch.\, III, Theorem 1.1, 1.2, and Lemma 2.1.) \qed

\smallskip\noindent

{\it Proof of Theorem 2}.
Since $\partial\Omega$ is of class $C^{2}$ there exists $0<\rho_0\leq r_1$ so that $\Omega$ satisfies the exterior sphere condition at each $\by\in\partial\Omega$  with a ball $B_{\rho_0}$ of radius $\rho_0$. It follows that for each  $\bx\in\Omega$ and $s>0$,  $B_s(\bx)\cap\Omega$ also satisfies the exterior sphere condition at each $\by\in\partial (B_s(\bx)\cap\Omega)$ with a ball $B_{\rho_0}$.

Assume the hypotheses of Theorem 2. Let  $\bo\in\partial\Omega$, $\bx_0\in B_r(\bo)\cap\Omega$,
 $0<\rho\leq r,\ {\rho\over 2}\leq s\leq\rho,\text{ and using Definition 2.1 set}$ $$Q_s(\bx) := Q_{B_s(\bx_0)}(\bx).$$

Just as in the proof of Theorem 1, we have
\begin{eqnarray}
&&\int_{B_{\rho/2}(\bx_0)\cap\Omega} [\alpha_1 |\nabla Q|^2+f_b (Q)]d\bx\leq {\sF}(Q; B_{2\rho}(\bx_0)\cap\Omega)+C_0\rho^2\\
&\leq& {\sF}(Q_s;B_{2\rho}(\bx_0)\cap\Omega)+C_0\rho^2\nonumber\\
&\leq&\int_{B_s(\bx_0)\cap\Omega} [\alpha_2 |\nabla H|^2+f_b( H)]d\bx+{\sF}(Q;(B_{2\rho}(\bx_0)\cap\Omega)\backslash B_s(\bx_0))\nonumber\\
&+& C_1\rho^2.\nonumber
\end{eqnarray}
Assume for now that $ H$ is separated from $\partial\sM$ on $\partial(B_s(\bx_0)\cap\Omega)$.
Then $f$ is $C^\infty$ on a neighborhood of $H|_{B_s(\bx_0)\cap\Omega}$ and
\[
 \Delta f(H)\geq 0\text{ on }B_s(\bx_0)\cap\Omega.
\]
We next solve
\begin{eqnarray*}
\Delta w=1&&\qquad\text{in}\qquad B_s(\bx_0)\cap\Omega,\\
w=0&&\qquad\text{on}\qquad \partial (B_s(\bx_0)\cap\Omega)
\end{eqnarray*}
and note that since $B_s(\bx_0)\cap\Omega$ satisfies the exterior sphere condition at each boundary point as described above we can  construct a barrier function relative to each point in $\partial (B_s(\bx_0)\cap\Omega)$ so as to ensure that the normal derivative
\[|\partial_\nu w(\by)|\leq C_2 s\text{ for each } \by\in\partial (B_s(\bx_0)\cap\Omega) \text{ for which } \partial_\nu w(\by) \text{ exists.}\]
To see this let $B_{\rho_0}(\bz)$ be the exterior sphere associated with $\by$ (so that $\overline B_{\rho_0}(\bz)\bigcap(\overline{B_s(\bx_0)\cap\Omega}) =\{\by\}$) and set $\eta=\eta(\bx)=|\bx-\bz|$. Here we can take $\rho_0=\rho_0(\partial\Omega)>0$ fixed and  assume without loss of generality that $s<\rho_0$. Then
 $ B_s(\bx_0)\cap\Omega\subset B_{(2s+\rho_0)}(\bz)\setminus B_{\rho_0}(\bz)$ and the function
 \[
 g(\eta(\bx))\equiv g(\eta)=\rho_0^2(3+\frac{6s}{\rho_0})\ln(\frac{\rho_0}{\eta})+3\rho_0(\eta-\rho_0)
 \]
 satisfies
 \[
g(\rho_0) =0,\, g(\eta)\leq 0 \text{ and } \Delta g\geq 1 \text{ for }\rho_0\leq\eta\leq\rho_0+2s.
\]
This can then be used as a barrier so that $|\partial_\nu w(\by)|\leq|g'(\rho_0)|=6s.$
The set $B_s(\bx_0)\cap\Omega$ is such that $w\in H^2(B_s(\bx_0)\cap\Omega)$ and we have
\begin{eqnarray*}
\int_{B_s(\bx_0)\cap\Omega} f( H)d\bx&=&\int_{B_s(\bx_0)\cap\Omega} w \Delta f( H)d\bx+\int_{\partial(B_s(\bx_0)\cap\Omega)} f(Q)\partial_\nu wd\bs\\
&\leq&\int_{\partial (B_s(\bx_0)\cap\Omega)} f(Q)\partial_\nu wd\bs.
\end{eqnarray*}
It follows that
\begin{equation}
\int_{B_s(\bx_0)\cap\Omega} f_b( H)d\bx\leq C_3 s\int_{\partial(B_s(\bx_0)\cap\Omega)} f_b(Q)d\bs+C_4 s^2.
\end{equation}
We remark that one needs $H(\bx)$ and $f(H(\bx))$ in $H^2(B_s(\bx_0)\cap\Omega)$ to directly apply Green's identity as above, however these functions are only in $H^1(B_s(\bx_0)\cap\Omega)$. We can circumvent this problem by using an approximation. We assume without loss of generality that $s$ is sufficiently small so that $B_s(\bx_0)\cap\Omega$ is strictly star-like with respect to a point $\bp\in B_s(\bx_0)\cap\Omega$ and set $H_\lambda(\bx)=H(\bp +\lambda(\bx-\bx_0))$. Then for $0<\lambda<1$ we have $H_\lambda$ and $ f(H_\lambda)$ in $C^\infty(\overline{(B_s(\bx_0)\cap\Omega)}$, such that $H_\lambda$ is harmonic, $f(H_\lambda)$ is subharmonic, and such that $H_\lambda\to H, f(H_\lambda) \to f(H)$ in $H^1(B_s(\bx_0)\cap\Omega)$ as $\lambda\uparrow 1$. We can then apply Green's identity for $\lambda<1$ and obtain (2.7) in the limit as $\lambda\uparrow 1$.

We next replace the condition that $ H$ is separated from $\partial\sM$ on $\partial(B_s(\bx_0)\cap\Omega)$ with the condition that $\ds\int_{\partial(B_s(\bx_0)\cap\Omega)}f_b(Q)d\bs<\infty$.
This is done just as before by considering $\tau H$ for $0<\tau<1$ for which (2.7) is valid and letting $\tau\uparrow 1$.
Combining (2.6) and (2.7) we have
\begin{eqnarray*}
\int_{B_{\rho/2}(\bx_0)\cap\Omega} [\alpha_1 |\nabla Q|^2 + f_b (Q)]d\bx&\leq&\alpha_2\int_{B_s(\bx_0)\cap\Omega} |\nabla H|^2d\bx\\
&+&C_5 s\int_{\partial(B_s(\bx_0)\cap\Omega)} f_b(Q)d\bs\\
&+& {\sF}(Q;(B_{2\rho}(\bx_0)\backslash B_s (\bx_0))\cap\Omega)+C_6\rho^2.
\end{eqnarray*}
We write $\partial(B_s(\bx_0)\cap\Omega)=(\partial B_s(\bx_0)\cap\Omega) \cup( B_s(\bx_0)\cap\partial\Omega).$ Then using (1.5) we have
\[\int_{\partial(B_s(\bx_0)\cap\Omega)} f_b(Q)d\bs\leq\int_{\partial B_s(\bx_0)\cap\Omega} f_b(Q)d\bs +C_7\]
and it follows that

\begin{eqnarray*}
\int_{B_{\rho/2}(\bx_0)\cap\Omega} [\alpha_1 |\nabla Q|^2 + f_b (Q)]d\bx&\leq&\alpha_2\int_{B_s(\bx_0)\cap\Omega} |\nabla H|^2d\bx\\
&+&C_5 s\int_{\partial B_s(\bx_0)\cap\Omega} f_b(Q)d\bs\\
&+& {\sF}(Q;(B_{2\rho}(\bx_0)\backslash B_s (\bx_0))\cap\Omega)+C_8\rho.
\end{eqnarray*}
At this point we argue just as in the proof of Theorem 1 arriving at
\begin{eqnarray*}
\beta_1\int_{B_{\rho/2}(\bx_0)\cap\Omega} [|\nabla Q|^2+f_b (Q)]d\bx&\leq&\beta_2\int_{(B_{2\rho}(\bx_0)\cap\Omega)\backslash B_{\rho/2}(\bx_0)}
[|\nabla Q|^2+f_b(Q)]d\bx\\
&+& C_9 \rho
\end{eqnarray*}
for fixed positive constants $\beta_1,\beta_2$, and $C_9$, where here they depend on $Q_0$ as well.
We next extend $Q$ in a neighborhood of $\Omega'$ of $\overline\Omega$ so that
\[
Q\in  H^1(\Omega';\overline\sM)\cap C^{0,1} (\Omega'\setminus\Omega;\overline\sM)
\]
and extend $f_b\equiv 0$ in $\Omega'\backslash\overline\Omega$.
We then have
\begin{eqnarray*}
\beta_1\int_{B_{\rho/2}(\bx_0)}[|\nabla Q|^2+f_b]d\bx&\leq&\beta_2\int_{B_{2\rho}(\bx_0)\backslash B_{\rho/2}(\bx_0)}[|\nabla Q|^2+f_b]d\bx\\
&+&C_{10}\rho
\end{eqnarray*}
for all $\bx_0$ in a neighborhood of $\partial\Omega$ and all $\rho$ sufficiently small.
Our assertion follows just as in the proof for Theorem 1.\qed

\noindent
{\it Proof of Corollary 2}.

To prove the interior regularity result, let $B$ be an open disk in $\Omega$ and $Q$ a local minimizer for $\cal F$ in $H^1(B;\overline \sM)$.
Assume $\overline{B_{4r}(\bold o)}\subset B\backslash \Lambda(Q)$.
As outlined in Section 1 it is enough to consider solutions to the quasilinear Euler--Lagrange equation for ${\cal F'}$ from (1.7) satisfying the strong ellipticity condition (1.8) in $\overline{B_{4r}(\bold o)}$.
It suffices to show that our solutions have H\"{o}lder--continuous first derivatives in $B_r(\bold o)$; higher regularity then follows using techniques from linear theory \cite{G,C}.
Note that $Q=\tilde Q(\bz(\bx))$, $f_b(\tilde Q(\bz))$ is a $C^\infty$ function if $\bz$ in a neighborhood of the range of $\bz(\bx)$ for $\bx$ in $\overline{B_{4r}(\bold o)}$, and the coefficients of $f_{ld}$ are polynomials in $\bz$.
By Theorem 1, $\bz(\bx)\in C^\sigma (\overline{B_{4r}(\bold o)})$ and we can apply the result from \cite{G}, Ch.~VI, Proposition 3.1 asserting that in two space dimensions a continuous weak solution is in $H^{2,p}$ for some $p>2$, and thus, first order derivations are H\"older continuous.
The argument given there proves an interior estimate.

To prove regularity near the boundary, we assume $Q$ is a local minimizer in $\Omega\cap B_{4r}(\bo)$, where $\bo\in\partial\Omega$ and $r<r_1$, $B_{4r}(\bold o)\cap\partial\Omega$ is of class $C^{k,\alpha}$ for some $k\geq 2$ and $\alpha\in (0,1)$, $Q_0\in C^{k,\alpha} (B_{4r}(\bold o)\cap\partial\Omega)$, and $\overline{B_{4r}(\bold o)}\cap \Lambda(Q)=\emptyset$.
As before, $Q=\tilde Q(\bz(\bx))$ and $f_p(\tilde Q(\bz))$ is a $C^\infty$ function if $\bz$ on a neighborhood of the range of $\bz(\bx)$ for $\bx\in\overline{B_{4r}(\bold o)\cap\Omega}$.
Here it is enough to prove that $Q$ has H\"older--continuous first derivatives in $B_{r\over 2} (\bold o)\cap\overline\Omega$.
One can derive a reverse--H\"older type inequality on sets $B_{2r} (\bx_0)\cap\Omega$ where $\bx_0\in B_{2r} (\bold o)\cap\partial\Omega$ in the spirit of that given in \cite{G}, Ch.~V, Sec. 2, and then apply the same argument to prove the estimate near the boundary asserting that the solution is in $H^{2,p} (B_{r\over 2} (\bold o)\cap\Omega)$.
The details for this inequality are given in the Appendix.
As before, higher regularities follows from linear theory.\qed

\section{Analysis of $\Lambda(Q)$}

In Section 2 we proved that a local minimizer on a ball $B_{4\overline r}(\bo)\subset\Omega$ satisfies
\begin{equation}
\|Q\|_{C^\delta(\overline{B_{\overline r}(\bo)})}\leq A_1
\end{equation}
for a fixed $0<\delta<1$  by establishing the Morrey--type estimate
\begin{equation}
\int_{B_\rho(\by)}[|\nabla Q|^2+f_b(Q)]d\bx\leq A_2\rho^{2\delta}
\end{equation}
for all $\by\in B_{\overline r}(\bo)$ and $0<\rho \leq \frac{\overline r}{2}$.
This was done by constructing comparison functions using harmonic replacements for each component of $Q$.
In this section we consider $f_e$ satisfying (1.9) and (1.10), and prove that if $Q$ is a finite energy local minimizer for $\sF$ on an open set $E$ then $\Lambda(Q)\cap E=\emptyset$.
We do this by refining the notion of replacement in this case so as to prove that (3.2) holds for any $0<\delta<1$.
We are then able to carry this one step further to get
\[
\int_{B_\rho(\bo)}f_b(Q)d\bx\leq A_3\rho^2
\]
for all $0<\rho$ sufficiently small.
If $\bo\in\Lambda(Q)$ we know that $\underset{\by\to \bo}{\lim} f_b (Q(\by))=\infty$ and if $\Lambda(Q)\cap E$ is nonempty, this leads to a contradiction.

\noindent
{\it Proof of Theorem 3}.
Assume $f_e$ satisfies (1.9) and (1.10) and let $Q$ be a local minimizer for $\sF$ on $B_{4\overline r}(\bo)\subset\Omega$. Assume without loss of generality that $\bo\in\Lambda (Q)$.
We associate with $f_e$ and any $V\in \overline\sM$ the quadratic form corresponding to
\[
L(V)[\nabla g]=\sum^2_{i,j=1} [L_1\ \delta_{ij}+L_5\ V_{ij}]g_{x_i}g_{x_j}.
\]
Using (1.10) we have that there exist constants $\mu_1,\mu_2>0$ so that
\begin{eqnarray*}
\mu_1 |\zeta|^2\leq L(V)[\zeta]&\leq&\mu_2 |\zeta|^2\text{ for all }\zeta\in\bR^2\\
&&\text{ and }V\in\overline\sM.
\end{eqnarray*}
Let $\by=A\bx$ be a linear nonsingular change of variables in $\bR^2$.
Given $V(\bx)\in\overline\sM$ set $\hat V(\by)=V(A^{-1}(\by))$ and define the quadratic form
\[
\hat L(\hat V(\by))[\zeta]:=L(V(\bx))[A^T\zeta]=\sum^2_{i,j=1} c_{ij}(\hat V(\by))\zeta_i\zeta_j.
\]
Here
\[
c_{ij}(\hat V(\by))=\big[L_1AA^t+A\hat W(\by)A^t\big]_{ij} \text{ for } 1\leq i,j\leq 2
\]
where $\hat W(\by)$ is the $2\times 2$ matrix whose $lk$ entry is the $lk$ entry of $\hat V(\by)$. Thus $c_{ij}(\hat V(\by))$ are affine functions of $\hat V_{lk}(\by)$ for $1\leq l,k\leq 2$.
Choose and fix $A$ so that $c_{ij}(\hat Q(\bo))=c_{ij}(Q(\bo))=\delta_{ij}$.
Note that  $c_{ij}=c_{ji}$ and $\tilde\mu_1|\zeta|^2\leq c_{ij}\zeta_i\zeta_j\leq\tilde\mu_2 |\zeta|^2$ where
$0<\tilde\mu_1\leq\tilde\mu_2$ are constants independent of $\hat V(\by)$ for $\hat V\in \overline\sM$.
Set $\sE=A [B_{\overline r} (\bo)]$ and let $r\leq \overline R$ where $\overline R$ satisfies $B_{4\overline R}(\bo)\subset\sE$.
By the chain rule, the energy $\sF$ transforms to
\[
\hat{\sF} (\hat V;{\sE})=\int_{\sE} [\hat f_e (\hat V,\nabla\hat V)+f_b (\hat V)]|\det A^{-1}|d\by
\]
where
\[
\hat f_e(\hat V,\nabla\hat V)=\sum^3_{l,k=1}\hat L(\hat V)[\nabla\hat V_{lk}]+\bb\cdot \nabla\hat V
\]
such that
\[
\bb=[b_{lkm}(\hat V)], \,\ b_{lkm}(\hat V)=\sum_{i,j=1}^2 b_{lkm}^{ij}\hat V_{ij},
\]
and $\{ b_{lkm}^{ij}\}$ are all constants that are uniformly bounded independent of $\hat V$ for $\hat V\in \sM$. For $\overline{\by}\in B_r(\bo)$ let $\sL(\overline{\by})$ be the constant coefficient scalar elliptic operator
\[
{\sL}(\overline{\by})[\hat g]=\sum^2_{i,j=1} c_{ij}(\hat Q(\overline{\by}))\hat g_{y_i y_j}.
\]
Let $0<s\leq r$. Then $B_s (\overline{\by})\subset\sE$ and we can construct the $\sL$--replacement comparison function $\hat Q_s({\by}):=\hat Q_{B_s(\overline{\by})}({\by})$ such that for each $1\leq\ell,\ k\leq 3$
\begin{equation}
{\sL}(\overline{\by})[\hat {Q_s}_{lk}]=0\qquad\text{on }B_s(\overline{\by}),
\end{equation}
\begin{equation}
\hat {Q_s}_{lk}=\hat Q_{lk}\qquad\text{ on }\partial B_s(\overline{\by}),
\end{equation}
\begin{equation*}
\hat {Q_s}_{lk}=\hat Q_{lk}\qquad\text{ on }\sE\backslash B_s (\overline{\by}).
\end{equation*}
In the following estimates the constants $M_i$ depend on $A_2$ and $\delta$ from (3.2), however they are independent of $s$ and $r$ for $0<s \leq r\leq r_0$.
Since homogeneous solutions to $\sL(\overline{\by})$ satisfy the strong maximum principle $\hat Q_s$ has the same properties as the harmonic replacements used earlier.
We use two other features of $\hat Q_s$.
Since $\hat{Q_s}_{lk}$ minimizes
\[
\int_{B_s(\overline{\by})}c_{ij}(\hat Q(\overline{\by})) g_{y_i} g_{y_j} d\by
\]
among $g\in  H^1 (B_s (\overline{\by}))$ satisfying $g=\hat Q_{lk}$ on $\partial B_s (\overline{\by})$ it follows that
\begin{equation}
\int_{B_s(\overline{\by})} |\nabla \hat Q_s|^2d\by\leq {M}_1\int_{B_s(\overline{\by})} |\nabla\hat Q|^2d\by
\end{equation}
where $M_1=M_1 (\tilde\mu_1,\tilde\mu_2)$.
Second, from the maximum principle and (3.1) we have
\begin{equation}
\underset{B_s(\overline{\by})}{osc}\ {\hat Q_{s_{ij}}}\leq \underset{\partial B_s(\overline{\by})}{osc}\  {\hat Q_{s_{ij}}}=\underset{\partial B_s(\overline{\by})}{osc}\ {\hat Q_{{ij}}}\leq {M}_2 s^\delta \text{ for each } i \text{ and } j.
\end{equation}
Just as before if $\hat Q_s|_{\partial B_s(\overline{\by})}$ is separated from $\partial\sM$ then the function $f(\hat Q_s(\overline{\by}))$ is a classical $\sL$--subsolution on $B_s(\overline{\by})$.
Set
\begin{equation*}
u(\by)={1\over 4} (|\by-\overline{\by}|^2-s^2)
\end{equation*}
 and let $w(y)$ solve
\begin{eqnarray*}
\sL(\overline{\by})[w]&=&1\qquad\text{ on }B_s(\overline{\by}),\\
w&=& 0\qquad\text{ on }\partial B_s (\overline{\by}).
\end{eqnarray*}
Then
\begin{eqnarray*}
\sL(\overline \by)[w-u]&=&1-\frac{1}{2}\sum_{i=1}^2c_{ii}(\hat Q(\overline\by))\\
&=&\frac{1}{2}\sum_{i=1}^2\big(c_{ii}(\hat Q(\bo))-c_{ii}(\hat Q(\overline\by)\big)=C
\text{ on }B_s(\overline{\by})
\end{eqnarray*}
where $C$ is  constant on $B_s(\overline{\by})$ such that
\[
|C|\leq M_3|\hat Q({\bo})-\hat Q(\overline\by)|\leq M_4 r^\delta.
\]
 Here we have used (3.1) and the formulas for the $c_{ij}$.
It follows that
\[
|\nabla w-\nabla u|\leq M_5s r^\delta\qquad\text{ on }\overline{B_s}(\overline{\by}).
\]
In particular
\[
|\nabla w(\by)-\frac{s}{2}\nu|\leq M_6 s r^\delta \text{ for }\by\in\partial B_s(\overline{\by}) \text{ and } \nu=\frac{(\by-\overline\by)}{s}.
\]
Also, since $f(\hat Q_s)$ is a subsolution for $\sL(\overline \by)$, we have
\begin{eqnarray}
\int_{B_s(\overline{\by})} f(\hat Q_s)d\by&=&\int_{B_s(\overline{\by})} f(\hat Q_s)\sL (\overline{\by})[w]d\by\\
&\leq&\int_{\partial B_s (\overline{\by})} f(\hat Q_s)\nu_\ell c_{lk}(\hat Q(\overline{\by}))w_{y_k}d\bs.\nonumber
\end{eqnarray}
We need to express this in terms of $f_b$. In order to do this set
\[
\mathfrak{I}=\int_{\partial B_s (\overline{\by})} (-\kappa|\hat Q_s|^2+b_0)\nu_\ell c_{lk}(\hat Q(\overline{\by}))w_{y_k}d\bs.
\]
Since $c_{lk}( Q)$ is a fixed affine function such that $c_{lk}( Q(\bo))=\delta_{lk}$ it follows that

\begin{equation}
|\nu_\ell c_{lk}(\hat Q(\overline{\by}))w_{y_k}-\frac{s}{2}|\leq M_7 sr^\delta.
\end{equation}
Adding $\mathfrak{I}$ to the right side of (3.7) then gives
\begin{equation}
\int_{\partial B_s (\overline{\by})} f_b(\hat Q_s)\nu_\ell c_{lk}(\hat Q(\overline{\by}))w_{y_k}d\bs\leq
s(\frac{1}{2}+M_{8}r^\delta)\int_{\partial B_s (\overline{\by})} f_b(\hat Q_s)d\bs.
\end{equation}
Also from (3.8)  we get
\[
|\mathfrak{I}-{s\over 2}\int_{\partial B_s}(-\kappa|\hat Q_s|^2+b_0)\,d\bs|\leq M_{9} s^2r^\delta.
\]

Next, using the fact that $\ds\int_{B_s} c\, d\by={s\over 2}\int_{\partial B_s} c\,d\bs$ if $c$ is a constant
and using  (3.6) we have the estimate
\[
|\mathfrak{I}-\int_{B_s (\overline{\by})} (-\kappa|\hat Q_s|^2+b_0)d\by|\leq M_{10}(s^{2+\delta}+s^2r^\delta)\leq M_{11} s^2r^\delta.
\]
Thus adding $\mathfrak{I}$ to the left side of (3.7) together with (3.9) gives
\[
\int_{B_s(\overline\by)} f_b (\hat Q_s)d\by\leq s({1\over 2}+M_8 r^\delta)\int_{\partial B_s (\overline{\by})} f_b (\hat Q)d\bs+M_{12} s^2 r^\delta.
\]
Just as before we can drop the assumption that $\hat Q|_{\partial B_s(\overline{\by})}$ is separated from $\partial\sM$ and replace it with the condition that $ f_b(\hat Q)\in L^1(\partial B_s (\overline{\by}))$.
Also as in Section 2, using $\hat Q_s$ as a comparison function and the above inequality we get
\begin{eqnarray*}
\int_{B_s(\overline{\by})} [\hat f_e (\hat Q,\nabla\hat Q)&+&f_b(\hat Q)]d\by\leq\int_{B_s(\overline{\by})} \hat f_e(\hat Q_s,\nabla\hat Q_s)d\by\\
&+& s({1\over 2}+M_8 r^\delta)\int_{\partial B_s(\overline{\by})}f_b(\hat Q)d\bs\\
&+& M_{12} s^2 r^\delta.
\end{eqnarray*}
We rewrite this as
\begin{eqnarray}
&&\int_{B_s(\overline{\by})} [\hat f_e (\hat Q(\overline{\by}),\nabla\hat Q)-\hat f_e (\hat Q(\overline{\by}),\nabla\hat Q_s)]d\by\\
&\leq&\int_{B_s(\overline{\by})} [(\hat f_e (\hat Q(\overline{\by}),\nabla\hat Q)-\hat f_e (\hat Q,\nabla\hat Q))\nonumber\\
&+& (\hat f_e (\hat Q_s,\nabla \hat Q_s)-\hat f_e (\hat Q(\overline{\by}),\nabla\hat Q_s))]d\by\nonumber\\
&-&\int_{B_s(\overline{\by})} f_b(\hat Q)d\by+s({1\over 2}+M_8 r^\delta)\int_{\partial B_s(\overline{\by})} f_b(\hat Q)d\bs+M_{12} s^2 r^\delta\nonumber
\end{eqnarray}
The left side of (3.10) equals
\begin{eqnarray}
&&\int_{B_s(\overline{\by})} c_{ij}(\hat Q(\overline{\by}) [\hat Q_{lk,y_i}\hat Q_{lk,y_j}-\hat Q_{slk,y_i}\hat Q_{slk,y_j}]d\by\\
&+&\int_{B_s(\overline{\by})}\bb (\hat Q(\overline{\by}))\cdot\nabla (\hat Q-\hat Q_s)d\by.\nonumber
\end{eqnarray}
Using (3.3) and (3.4) the first integral in (3.11) equals and satisfies
\begin{eqnarray*}
&&\int_{B_s(\overline{\by})}c_{ij}(\hat Q(\overline{\by}))(\hat Q_{lk,y_i}-\hat Q_{slk,y_i})(\hat Q_{lk,y_j}+\hat Q_{slk,y_j})d\by\\
&=&\int_{B_s(\overline{\by})}c_{ij}(\hat Q(\overline{\by}))(\hat Q_{lk,y_i}-\hat Q_{slk,y_i})(\hat Q_{lk,y_j}-\hat Q_{slk,y_j})d\by\\
&\geq&\tilde\mu_1\int_{B_s(\overline{\by})}|\nabla (\hat Q-\hat Q_s)|^2d\by.
\end{eqnarray*}
For the second integral in (3.11), note that $\bb(\hat Q(\overline{\by}))$ is constant and $\hat Q=\hat Q_s$ on $\partial B_s(\overline{\by})$.
Thus this integral is zero.

The first integral on the right side of (3.10) can be estimated using (3.1),(3.2),(3.5),(3.6) and the formula for $\hat f_e$.
We find that it is bounded by
\begin{eqnarray*}
&&M_{13}\Big( s^\delta\int_{B_s(\overline{\by})} [ |\nabla\hat Q|^2 + |\nabla\hat Q_s|^2]d\by+s^{1+\delta}(\int_{B_s(\overline{\by})}[|\nabla\hat Q|^2+|\nabla \hat Q_s|^2]d\by)^{1/2}\Big)\\
&&\leq M_{14} s^{3\delta}.
\end{eqnarray*}
We arrive at
\begin{eqnarray}
\tilde\mu_{1}\int_{B_s(\overline{\by})}&&|\nabla (\hat Q-\hat Q_s)|^2d\by\leq M _{15} (s^{3\delta}+s^2 r^\delta)\\
&-&\int_{B_s(\overline{\by})} f_b (\hat Q)d\by+s({1\over 2}+M_8 r^\delta)\int_{\partial B_s(\overline{\by})} f_b(\hat Q)d\bs.\nonumber
\end{eqnarray}
Thus we have
\begin{eqnarray}
\tilde\mu_1 \int_{B_s(\overline{\by})} &&|\nabla (\hat Q-\hat Q_s)|^2d\by\leq M_{15} (s^{3\delta}+s^2 r^\delta)\\
&-&\int_{B_s(\overline{\by})} f_b (\hat Q)d\by+{s\over m}\int_{\partial B_s(\overline{\by})} f_b(\hat Q)d\bs.\nonumber
\end{eqnarray}
Here $m=m(r_0)$ such that $0<m<2$ and can be taken as close to 2 as desired provided $r_0$ is taken sufficiently small to begin with.
Now set
\[
\gamma(s)=\int_{B_s(\overline{\by})} f_b (\hat Q)d\by.
\]
From (3.13) we have
\[
- M_{15} (s^{3\delta}+ s^2 r^\delta)\leq -\gamma (s)+{s\over m}\gamma' (s)
\]
for almost every $0<s\leq r$. In what follows the constants $C_i$ will depend on $r, \delta,$ and $\eta$.
Let $0<\eta<\text{min}(3\delta,m)$.
Dividing by $s^{1+\eta}$ and using the fact that $\gamma'(s)\geq 0$ we find

\[-M_{16}s^{\text{min}(3\delta,m)-\eta-1}\leq\big({\gamma(s)\over s^\eta}\big)'.
\]
Since $\text{min}(3\delta,m) >\eta$ we can integrate this to obtain
\[
{\gamma(s)\over s^\eta}\leq C_1\qquad\text{for }0<s\leq r.
\]
Thus
\[
\int_{B_s(\overline{\by})} f_b(\hat Q)d\by\leq C_1 s^\eta\text{ for } |\overline{\by}|\leq r, s\leq r,
\]
for $\eta<\min(3\delta, m)$.

We next get a corresponding, improved estimate for $|\nabla\hat Q|^2$.
Since $\gamma(s)$ is absolutely continuous on $[0,r]$, given $0<t\leq r$ we can select $s_0$ so that ${t\over 2}\leq s_0\leq t$ and $t\gamma' (s_0)\leq 2\gamma(t)$.
Then from (3.13) we get
\begin{equation}
\int_{B_{s_0}(\overline{\by})} |\nabla (\hat Q-\hat Q_{s_0}|^2d\by\leq C_2 (t^{3\delta} + t^\eta)\leq C_3 t^\eta.
\end{equation}
Since $\hat Q_{s_0}$ solves (3.3) and (3.4) we have that
\[
\int_{B_\rho(\overline{\by})} |\nabla\hat Q_{s_0}|^2d\by\leq C_4 ({\rho\over s_0})^2\int_{B_{s_0}(\overline{\by})}|\nabla\hat Q_{s_0}|^2d\by\text{ for }\rho\leq s_0.
\]
(See \cite{G}; Ch.~III.)
Thus using (3.5) and (3.14)
\[
\int_{B_\rho(\overline{\by})} |\nabla\hat Q|^2d\by\leq C_5 ({\rho\over s_0})^2\int_{B_{s_0}(\overline{\by})} |\nabla\hat Q|^2d\by+C_6 t^\eta\text{ for }\rho<s_0,
\]
and it follows that
\[
\int_{B_\rho(\overline{\by})} |\nabla \hat Q|^2d\by\leq C_7 ({\rho\over t})^2\int_{B_t(\overline{\by})}|\nabla\hat Q|^2d\by+ C_8 t^\eta
\]
for $0<\rho\leq t\leq r$, $|y|\leq r$, and $r\leq r_0$.
Then from \cite{G}; Ch.~III, Lemma 2.1 we find that
\[
\int_{B_t(\overline{\by})} |\nabla\hat Q|^2d\by\leq C_9 t^\eta\text{ for } |\overline{\by}|\leq r,\ t\leq r.
\]

Given $0<\theta<1$ set $\delta_1=\min ({5\delta\over 4},\theta)$.
We have strengthened (3.2) by showing that the inequality holds with $\delta$ replaced by $\delta_1$.
The argument can be repeated a finite number of times to obtain the estimate with $\delta$ replaced by $\theta$.
We return to (3.12) setting $\overline{\by}=\bo,s=r$, and $\delta=\theta$.
 Using the definition of $\gamma(r)$ we find that there exists $r_1>0$ so that
\begin{equation}
-2M_{15} r^{3\theta}\leq -\gamma (r)+({r\over 2}+M_8 r^{1+\theta})\gamma' (r)
\end{equation}
for almost every $0<r\leq r_1$.
We need $\theta>{2\over 3}$ and set $\theta={3\over 4}$.
Multiplying (3.15) by the integrating factor $2r^{-3} (1+2M_8 r^{3/4})^{5/3}$ we obtain
\[
-M_{16} r^{-{3\over 4}}\leq \big[r^{-2} (1+2M_8 r^{3/4})^{8/3}\gamma (r)\big]'.
\]
Integrating from $r$ to $r_1$ we get
\[
r^{-2}\int_{B_r(\bo)} f_b (\hat Q)d\by=r^{-2}\gamma (r)\leq C = C(r_1)
\]
for $0<r\leq r_1$.
We assumed that $\hat Q(\bo)\in\partial\sM$ implying that $\ds\lim\limits_{{\by}\to \bo}\ f_b (\hat Q({\by}))=\infty$ which is impossible.\qed

We can use the harmonic replacement method  to give a short proof of a local version of the Ball--Majumdar result.
We do this here in a slightly more general setting.
Let $\sK\subset \bR^m$ be an open, bounded, convex set where $m\geq1$.
Consider $f\colon \sK\to\bR$ such that $f(\bp)=\tilde f(\bp)-\kappa|\bp|^2$ where $\kappa \geq 0$, $\tilde f\in
C^\infty(\sK)$, convex, and such that $\lim\limits_{\bp\to\partial \sK}\tilde f(\bp)=\infty$.

\begin{theorem}Let $\Omega$ be a domain in $\bR^n$ for $n\geq1$, $\gamma>0$, and let $u(\bx)\in
W^{1,2}(\Omega;\overline \sK)$ be a local minimizer for
\begin{equation*}
\sF_1(w)=\int_\Omega [\gamma|\nabla w|^2+f(w)]\ d\bx
\end{equation*}
such that $\sF_1(u)<\infty$.
Then $u\in C^\infty(\Omega)$  and $u(\bx)\in \sK$ for each $\bx\in\Omega$.
\end{theorem}

\begin{proof}Let $B=B_r(\bx_0)\subset\Omega$ and let $u_B(\bx)$ be defined to have harmonic components and satisfy $u_B=u$ on $\partial B$.
It follows from  the maximum principle and the fact that $\sK$ is convex that  $u_B(\bx)\in\overline
\sK$ for $\bx\in B$.
We can then compare $u$ and $u_B$,
\begin{equation}
\int_B [\gamma |\nabla u|^2+f(u)]\ d\bx\leq\int_B[\gamma|\nabla u_B|^2+f(u_B)]\ d\bx.
\end{equation}
Since $u_B$ is harmonic we have that
\begin{equation*}
\gamma\int_B |\nabla (u-u_B)|^2 d\bx=\gamma\int_B (|\nabla u|^2-|\nabla u_B|^2) d\bx.
\end{equation*}
Using the fact that $\sK$ is  bounded and the Poincar\'e and H\"{o}lder inequalities we have that
\begin{eqnarray*}
&&\int_B \kappa| |u|^2-|u_B|^2| d\bx\leq C_0\int_B |u-u_B|d\bx\\
&&\leq C_1 r\int_B |\nabla (u-u_B)| d\bx\\
&\leq& {\gamma\over2}\ \int_B |\nabla (u-u_B)|^2 d\bx+ C_2 r^{n+2}.
\end{eqnarray*}
Next, just as in the proof for Theorem 1 we have that
\[
\int_B \tilde f(u_B) d\bx\leq {r\over n}\ \int_{\partial B}\tilde f(u) d\bs.
\]
Combining these facts with (3.16) we get
\[
-C_3 r\leq {d\over dr} (r^{-n}\int_{B_r(\bx_0)}\tilde f(u) d\bs).
\]
Given $\delta>0$ set $\Omega_\delta=\{\bx\in\Omega\colon$ dist$(\bx,\partial\Omega)\geq \delta\}$.
Then just as in the proof of Theorem 3 we find that for $\bx\in\Omega_\delta$ and $r<\delta/2$ that
\[
|B_r(\bx)|^{-1} \int_{B_r(\bx)}\tilde f(u) d\bx \leq C_4 < \infty
\]
where $C_4=C_4(\delta,\int_\Omega \tilde f (u)\,d\bx)$.
It follows that $\tilde f(u(\bx))\leq C_4$ almost everywhere on $\Omega_\delta$.
From the structure of $\tilde f$ we see that there is a $\beta>0$, $\beta=\beta(C_4)$ so that
\[
u(\bx)\in \sK_\beta=\{\bp\in \sK\colon\text{ dist}(\bp,\partial \sK)\geq\beta\}
\]
for almost every $\bx\in\Omega_\delta$.
Due to this we can take smooth variations with compact support about $u$ for $\sF_1$ and as a consequence  $u$ satisfies
 the energy's Euler--Lagrange
equation.
Since $f(\cdot)\in C^\infty(\sK)$  it follows that $\Delta u\in L^\infty_{\loc} (\Omega)$ and we have
that $u\in C^{1,\alpha}(\Omega)$ for $0<\alpha<1$. Higher regularity follows from elliptic estimates.
\end{proof}

 We can apply Theorem 4 to tensor-valued functions as in Section 1.

 \begin{corr} Let $D$ be a bounded domain in $\mathbb{R}^3$ and let $Q\in H^1(D;\overline\sM)$ be a local minimizer for $\sF(V;D)=\int_D[L_1|\nabla V|^2 +f_b(V)] d\bx$  such that $\sF(Q;D)<\infty$. Then $Q\in C^\infty(D)$ and $Q(\bx)\in\sM$ for each $\bx\in D$ .
 \end{corr}

 \begin{proof} Let $\{E_1,\dots,E_5\}$ be an orthonormal basis for $\sS_0$ and parameterize $W\in \sS_0$ by $W(u)=\sum_{i=1}^5u_iE_i \text{ for } u\in\mathbb{R}^5$. Then  $\sK:=\{u\in\mathbb{R}^5:W(u)\in\sM\}$ is a bounded, open, and convex. For $V\in H^1(D;\overline\sM)$ we can write
  \[L_1|\nabla V(\bx)|^2+f_b(V(\bx))= L_1|\nabla u(\bx)|^2+f_b(V(u(\bx)))\]

 and the assertions follow from Theorem 4.
 \end{proof}

\appendix\section{   Appendix}

In this section we derive a reverse H\"older inequality that when combined with the proof from \cite{G};  Ch.~ VI, Proposition 3.1  demonstrates that local minimizers are regular near $\partial\Omega\cap\Lambda (Q)$.
Let $\mathbf x_0\in\partial\Omega\backslash\Lambda (Q)$.
Fix $r>0$ so that $\overline{B_{4r}}(\mathbf x_0)\cap\Lambda(Q)=\emptyset$ and take $r$ sufficiently small so that $B_{4r} (\mathbf x_0)\cap\Omega$ is diffeomorphic to a half--disk.
From what has been proved we can assume that $Q=Q(\bold z)$ as in (1.6), where $\bold z\in C^2 (B_{2r} (\mathbf x_0)\cap\Omega)\cap C^\sigma ({B_{2r}}(\mathbf x_0)\cap\overline\Omega)\cap H^1 (B_{2r}
(\mathbf x_0)\cap\Omega)$ satisfying the Euler--Lagrange equation for $\sF'(\cdot)$,
\begin{eqnarray}
 &&\partial_{x_k}(A_{ij}^{lk}(\mathbf z)\partial_{x_l} z^i)+B^{kl}_{imj}(\mathbf z)\partial_{x_l} z^i \partial_{x_k} z^m+ C^l_{ij} (\mathbf z)\partial_{x_l} z^i\\
&+& D_j (\mathbf z)=0\quad\text{ in }\quad B_{2r} (\mathbf x_0)\cap\Omega,\nonumber
\end{eqnarray}
with boundary condition
\begin{equation*}
\mathbf z=\mathbf z_0\qquad\text{ on }\quad B_{2r} (\mathbf x_0)\cap\partial\Omega,
\end{equation*}
here the coefficients are $C^1$ functions on a neighborhood $\sK$ of the range of $\mathbf z|_{B_{2r}(\mathbf x_0)\cap\Omega}$, such that
\[
\|\mathbf A\|_{C^1(\overline\sK)},\ \|\mathbf B\|_{C^1(\overline\sK)},\ \|\mathbf C\|_{C^1(\overline\sK)},\ \|\mathbf D\|_{C^1(\overline\sK)}\leq M < \infty,
\]
for $\lambda>0$ we assume that
\begin{equation*}
A_{ij}^{lk} (\mathbf z)\zeta^i_l\zeta^j_k\geq\lambda |\zeta|^2\text{ for }\zeta\in\bR^{5\times 2}\text{ and }\mathbf z\in \overline\sK,
\end{equation*}
and $\mathbf z_0\in C^2 (\partial\Omega)$.
Set
\[
\delta=\sup \|\mathbf z-\mathbf z_0\|_{C(B_{2r}(\mathbf x_0)\cap\Omega)}.
\]
We prove

\noindent
\begin{prop}
There exist constants $\delta_0$, $C>0$ depending on $\lambda,M,\mathbf z_0$, and $\Omega$ so that if $\delta\leq\delta_0$ then $D^2\mathbf z$, $|D\mathbf z|^2\in L^2(B_r (\mathbf x_0)\cap\Omega)$ and for $0<\rho\leq {r\over 2}$
\begin{eqnarray*}
{\int\!\!\!\!\!\!-}_{B_\rho(\mathbf x_0)\cap\Omega} [1+ |D^2\mathbf z| + |D\mathbf z|^2]^2 d\bx
\leq C\bigg( {\int\!\!\!\!\!\!-}_{B_{2\rho}(\mathbf x_0)\cap\Omega} [1+ |D^2\mathbf z|+|D\mathbf z|^2]\ d\bx\bigg)^2.
\end{eqnarray*}
\end{prop}
\medskip\noindent
{\it Proof}.
For simplicity we assume that $\mathbf x_0=\bO=(0,0)$ and
\[
B_{2r} (\bO)\cap\Omega=B_{2r}^+ =\{(x_1,x_2)\colon |\mathbf x|<2r,\ x_2>0\},
\]
such that
\[
B_{2r}(\bO)\cap\partial\Omega=\{(x_1,0)\colon |x_1|<2r\}.
\]
The general problem reduces to this case by flattening the boundary.
This introduces an explicit $\mathbf x$--dependence for the coefficients in (A.1) that does not affect the estimates, however the constants $C_i$ appearing below  will also depend on $\Omega$.

Let $\eta\in C_c^\infty (B_{2r} (\bO))$ such that $\eta=1$ on $B_r(\bO)$.
Let $h\neq 0$ and set
\[
\Delta_1^h u(\mathbf x)={1\over h}[u(x_1+h,x_2)-u(x_1,x_2)].
\]
We multiply (A.1) by the test functions $-\Delta_1^{-h}[\eta^4\Delta_1^h(\mathbf z-\mathbf z_0)]$ and $(\mathbf z-\mathbf z_0)|\Delta_1^h (\mathbf z-\mathbf z_0)|^2\eta^4$.
Integrating over $B_{2r}^+$ in each case yields for $\delta\leq\delta_0$
\begin{eqnarray*}
&&\int_{B_{2r}^+}[|\Delta_1^h D\mathbf z|^2+|\Delta_1^h\mathbf z|^2 |D\mathbf z|^2]\eta^4 d\bx\\
&\leq&C_1\int_{B_{2r}^+}[1+ |D\mathbf z|^2 + |\Delta_1^h\mathbf z|^2]\eta^4 d\bx\\
&+&C_2\int_{B_{2r}^+} |\Delta_1^h (\mathbf z-\mathbf z_0)|^2 |D\eta|^2\eta^2 d\bx
\end{eqnarray*}
where the $C_j$ depend only on $\delta_0, \lambda_1,M$, and $\|\mathbf z_0\|_{C^2}$.
Letting $h\to 0$ gives
\begin{eqnarray}
\hskip1truein &&\int_{B_{2r}^+} [|\partial_{x_1} D\mathbf z|^2+|\partial_{x_1}\mathbf z|^4 \eta^4 d\bx\\
&\leq&C_3\int_{B_{2r}^+} [1+ |D\mathbf z|^2]\eta^4 d\bx \nonumber\\
&+&C_4\int_{B_{2r}^+} |\partial_{x_1}(\mathbf z-\mathbf z_0)|^2 |D\eta|^2 \eta^2 d\bx. \nonumber
\end{eqnarray}
We next estimate $|\partial^2_{x_2}\mathbf z|^2\eta^4$ and $|\partial_{x_2}\mathbf z|^4\eta^4$.
From (A.1) we have the pointwise bound
\begin{equation}
|\partial^2_{x_2}\mathbf z|^2\leq C_5 (|\partial_{x_2}\mathbf z|^4+|\partial_{x_1} D\mathbf z|^2+
|\partial_{x_1}\mathbf z|^4 +1).
\end{equation}
For $\var>0$ set $B_{\var,r}^+=B_r^+\cap \{x_2>\var\}$.
Multiplying (A.1) by $(\mathbf z-\mathbf z_0)|\partial_{x_2}\mathbf z|^2\eta^4$, integrating over $B_{\var,2r}^+$ with $\delta<\delta_0$ yields
\begin{eqnarray}
\hskip1truein &&\int_{B_{\var,2r}^+}|\partial_{x_2}\mathbf z|^4 \eta^4 d\bx\leq C_6\delta\int_{B_{\var,2r}^+}|\partial_{x_2}^2\mathbf z|^2 \eta^4 d\bx\\
&+&C_7\int_{\{x_2=\var\}} |\mathbf z-\mathbf z_0\| D\mathbf z|^3 \eta^4 dx_1 \nonumber \\
&+&C_8\int_{B_{\var,2r}^+} [|\partial_{x_1} D\mathbf z|^2 + |\partial_{x_1}\mathbf z|^4 +1]\eta^4 d\bx \nonumber\\
&+&C_9\int_{B_{\var,2r}^+} |\bold z-\bold z_0|^4 |D\eta|^4 d\bx.\nonumber
\end{eqnarray}
We next estimate
\begin{eqnarray*}
{\rm (A.5)}\hskip1truein &&\int_{\{x_2=\var\}} |D\mathbf z|^3 \eta^4 dx_1\leq\int_{B_{\var,2r}^+} |\partial_{x_2} (|D\mathbf z|^3 \eta^4)| d\bx\\
&\leq&C_{10}\int_{B_{\var,2r}^+} [(\partial_{x_2} D\mathbf z|^2 + |D\mathbf z|^4]\eta^4 d\bx\\
&+&C_{11} \int_{B_{\var,2r}^+} |D\eta|^4 d\bx.
\end{eqnarray*}
We use this to estimate the second term on the right side of (A.4).
This together with the integral of (A.3) times $\eta^4$ over $B_{{\var,2r}^+}$ and (A.2) for $\delta\leq\delta_0$ yields
\[
\int_{B_{\var,2r}^+} [|D^2\mathbf z|^2+ |D\mathbf z|^4]\eta^4 d\bx\leq E<\infty
\]
where $E$ is independent of $\var$.

From this point we proceed using the fact that $|D^2\mathbf z| +|D \mathbf z|^2\in L^2 (B_r^+)$ and replacing $\eta$ by $\tilde\eta\in C^\infty (B_{2\rho} (\bO))$ such that $\tilde\eta=1$ on $B_\rho (\bO)$ for $\rho\leq \frac{r}{2}$.
Since $\underset{\var\to0}\lim |\bold z(x_1,\var)-\bold z_0(x_1)|=0$ uniformly in $x_1$ we now have
 \[\underset{\var\to 0}\lim \int_{\{x_2=\var\} } |\mathbf z-\mathbf z_0| |D\bold z|^3 \tilde\eta^4 dx_1=0.\]
Using this, (A.2), (A.3), and (A.4) give
\begin{eqnarray*}
{\rm (A.6)}\hskip1truein &&\int_{B_{2\rho}^+} [|D^2\mathbf z|^2 + |D\mathbf z|^4]\tilde{\eta}^4 d\bx\\
&\leq& C_{12}\int_{B_{2\rho}^+} [1+ |D\mathbf z|^2] \tilde{\eta}^4 d\bx\\
&+& C_{13}\int_{B_{2\rho}^+} |\partial_{x_1} (\mathbf z-\mathbf z_0)|^2 |D\tilde{\eta}|^2 \tilde{\eta}^2 d\bx\\
&+&C_{14}\int_{B_{2\rho}^+} |\mathbf z -\mathbf z_0|^4 |D\tilde{\eta}|^4 d\bx.
\end{eqnarray*}
We can assume that $|D\tilde\eta|\leq {C\over\rho}$.
Since $\mathbf z-\mathbf z_0$ and $\partial_{x_1} (\mathbf z -\mathbf z_0)$ vanish for $x_2=0$ we can apply the Sobolev--Poincar\'e inequality to the last two terms.
We get
\[
\int_{B_{2\rho}^+} |\partial_{x_1} (\mathbf z-\mathbf z_0)|^2 d\bx\leq C_{15}(\int_{B_{2\rho}^+} |D^2(\mathbf z-\mathbf z_0)|d\bx)^2,
\]
and
\[
\rho^{-2}\int_{B_{2\rho}^+} |\mathbf z-\mathbf z_0|^4 d\bx\leq C_{16}(\int_{B_{2\rho}^+} |D(\mathbf z-\mathbf z_0)|^2 d\bx)^2.
\]
Inserting these estimates into (A.6) gives our result.\qed

\end{document}